\documentclass[12pt,notitlepage]{amsart}
\usepackage{latexsym,amsfonts,amssymb,amsmath,amsthm} 
\newcommand{\mz}{\ensuremath{\mathbb Z}}
\newcommand{\mr}{\ensuremath{\mathbb R}}

\newcommand{\shortmod}{\ensuremath{\negthickspace \negthickspace \negthickspace \pmod}}

\newcommand{\half}{\ensuremath{ \frac{1}{2}}}

\newcommand{\intR}{\int_{-\infty}^{\infty}}

\newcommand{\sumstar}{\sideset{}{^*}\sum}
\newcommand{\leg}[2]{\left(\frac{#1}{#2}\right)}
\newcommand{\e}[2]{e\left(\frac{#1}{#2}\right)}

\newcommand{\sumflat}{\sideset{}{^\flat}\sum}
\newcommand{\lam}{\lambda}

\newcommand{\Lam}{\Lambda}

\theoremstyle{plain}		
	\newtheorem{mytheo}{Theorem}[section]
	
	\newtheorem{myprop}[mytheo]{Proposition}
	\newtheorem{mycoro}[mytheo]{Corollary}
     \newtheorem{mylemma}[mytheo]{Lemma}

\theoremstyle{remark}

\numberwithin{equation}{section}
\begin{document}
\title{The second moment of quadratic twists of  modular   $L$-functions}
\author{K. Soundararajan}
\address{Stanford University\\
450 Serra Mall, Building 380\\
Stanford, CA 94305-2125}
 \email{ksound@math.stanford.edu}

\author{Matthew P. Young}
\address{Department of Mathematics \\
	  Texas A\&M University \\
	  College Station \\
	  TX 77843-3368 \\
		U.S.A.}
\email{myoung@math.tamu.edu}
\thanks{ The authors are partially supported by grants from the NSF (DMS-0500711  and DMS-0758235)}

\maketitle

\section{Introduction}

\noindent   The family of quadratic twists of a modular form has 
received much attention in recent years.   Motivated 
by the Birch-Swinnerton-Dyer conjectures, we seek an understanding of the central values of the associated $L$-functions, 
and while this question has been investigated extensively, much remains 
unknown.   One important theme in this area concerns the moments 
of these central $L$-values.   Thanks to the work of Keating and Snaith \cite{KS}
there are now widely believed conjectures for the asymptotics of 
such moments, but only the asymptotic for the first moment has been proved (see 
\cite{BFH, I, MM}).  
In this paper we establish two results 
concerning the second moment of the central $L$-values. 
   Unconditionally we obtain a 
lower bound for the second moment which matches precisely the conjectured 
asymptotic formula.  Upon assuming the truth of the Generalized Riemann 
Hypothesis, we establish the conjectured asymptotic formula.

To state our results we need some notation.  For simplicity we 
shall work with modular forms of full level but our work can 
be extended to congruence subgroups.  
Let $f$ be a modular form of weight $\kappa$ for the full modular group and suppose 
that $f$ is an eigenfunction of all the Hecke operators.  We 
write the Fourier expansion of $f$ as
\begin{equation*}
f(z) = \sum_{n=1}^{\infty} \lam_f(n) n^{\frac{\kappa-1}{2}} e(nz),
\end{equation*}
with $\lam_f(1)=1$, and $f$ has been normalized so that Deligne's bound gives $|\lam_f(n)| \le d(n)$  
for all $n$, where $d(n)$ denotes the number of divisors of $n$.    
The $L$-function associated to $f$ is 
\begin{equation*}
L(s,f) = \sum_{n=1}^{\infty } \frac{\lam_f(n)}{n^s} = \prod_p \left( 1- \frac{\lam_f(p)}{p^s} + \frac{1}{p^{2s}} \right)^{-1},
\end{equation*}
which converges absolutely for Re$(s)>1$, extends analytically 
to the entire complex plane, and satisfies the functional equation 
\begin{equation*}
\Lam(s,f) = (2\pi)^{-s} \Gamma(s+\tfrac{\kappa-1}{2}) L(s,f) = i^\kappa
\Lam(1-s,f).
\end{equation*}

Let $d$ denote a fundamental discriminant, and $\chi_d(\cdot) = \leg{d}{\cdot}$
denote the primitive quadratic character of conductor $|d|$.  Let
$f\otimes \chi_d$ denote the twist of $f$ by the character $\chi_d$, 
and $L(s,f\otimes \chi_d)$ denote the twisted $L$-function 
\begin{equation*}
L(s,f\otimes \chi_d) = \sum_{n=1}^{\infty} \frac{\lam_f(n)}{n^{s}} 
\chi_d(n).
\end{equation*}
We set 
\begin{equation*}
\Lam(s,f\otimes \chi_d) = \leg{|d|}{2\pi}^s \Gamma(s+\tfrac{\kappa-1}{2}) 
L(s,f\otimes\chi_d), 
\end{equation*}
and then the twisted $L$-function satisfies the 
functional equation 
\begin{equation*}
\Lam(s,f\otimes\chi_d) = i^{\kappa} \epsilon(d) \Lam(1-s,f\otimes \chi_d), 
\end{equation*}
where $\epsilon(d) = \leg{d}{-1}$ is $1$ or $-1$ depending on whether 
$d$ is positive or negative.  
Note that the sign of the functional equation is negative if $\kappa \equiv 2\pmod 4$ and $d$ is 
positive, or if $\kappa\equiv 0\pmod 4$ and $d$ is negative, and in these cases 
the central $L$-value is zero.   


Below we shall use $\sum^*$ to denote a sum over square-free integers, and 
$\sum^{\flat}$ to denote a sum over fundamental discriminants.  
 
\begin{mytheo} 
\label{thm:lowerbound}  Let $\kappa \equiv 0\pmod 4$, and keep notations as above.  Then  
\begin{equation*}
\sumstar_{\substack{0< 8d \leq X \\ (d,2)=1}} L(\tfrac 12,f\otimes \chi_{8d})^2
\geq (c + o(1)) X \log{X}  
\end{equation*}
where 
$$
c= \frac{2}{\pi^2} L(1,\text{\rm sym}^2 f)^3 Z_2(0,0),
$$
and the value $Z_2(0,0)$ is defined in \eqref{eq:Z(u,v)} and \eqref{eq:Z(u,v)2}.
\end{mytheo}

In Theorem \ref{thm:lowerbound},  $c$ is the constant predicted by the Keating-Snaith conjectures, see \cite{KS, CFKRS}.   
For simplicity we have restricted attention to fundamental discriminants of 
the form $8d$, but we may also handle similarly all discriminants. 

Rudnick and Soundararajan \cite{RS1, RS2} have described a general method 
to obtain lower bounds for moments in families of $L$-functions.  
Their method would readily give a bound $\gg X\log X$ in Theorem \ref{thm:lowerbound}.  

The problem of estimating the second moment of quadratic twists 
of a modular form is comparable in difficulty with that of estimating 
the fourth moment of central values of quadratic Dirichlet $L$-functions.  
Analogously to Theorem \ref{thm:lowerbound} we could obtain a lower bound for that 
fourth moment which matches precisely the conjectured asymptotic 
formula; this was stated without proof in \cite{RS2}.  

\begin{mytheo} 
\label{thm:GRHasymp}  Suppose the Generalized Riemann Hypothesis 
holds for the family of $L$-functions $L(s,f\otimes \chi_{d})$ for 
all fundamental discriminants $d$, and also for $\zeta(s)$ and $L(s,\text{\rm sym}^2 f)$.  Then, 
for $\kappa\equiv 0\pmod 4$, and with $c$ being the constant in Theorem \ref{thm:lowerbound}, 
\begin{equation*}
\label{eq:1.8}
\sumstar_{\substack{ 0< 8d \le X \\ (d,2)=1}}  L(\tfrac 12,f\otimes \chi_{8d})^2 
= (c+o(1)) X \log X.
\end{equation*}
\end{mytheo}

Our method would allow us to get an error term in Theorem \ref{thm:GRHasymp} which is 
smaller than the main term by a small power of $\log X$.  If we consider 
a smooth sum over the discriminants $8d$ in place of the ``sharp cut-off" $0<8d \le X$ 
we would get an error term of $O(X (\log X)^{\frac 34+\varepsilon})$, see Section \ref{section:asymponGRH} below.  

The new input in Theorem \ref{thm:GRHasymp} arises from recent work of the first author \cite{SoundMoment}
on obtaining upper bounds for moments of $L$-functions assuming the 
GRH.  The work there will show that our second moment is (on GRH) bounded above 
by $X(\log X)^{1+\varepsilon}$.  To refine this to the asymptotic given here, we need to 
extend the technique in \cite{SoundMoment} to bound shifted moments of $L$-functions; 
for a precise statement see Theorem \ref{thm:GRHupperbound} below.  Similar upper bounds 
for analogous shifted moments for the Riemann zeta-function have been recently obtained 
by Chandee, see \cite{Chandee}.

As with Theorem \ref{thm:lowerbound} for simplicity we have restricted attention to fundamental discriminants of 
the form $8d$, and we may adapt our methods to cover other discriminants.   In families of 
discriminants where the sign of the functional equation is negative, we may adapt our 
methods to study the second moment of the derivative of the $L$-function at $\frac 12$.   
Further, we may adapt the technique described here to obtain an asymptotic 
formula for the fourth moment of quadratic Dirichlet $L$-functions, conditional on the 
GRH.    We note here the recent work of Bucur and Diaconu \cite{BD} which treats 
the fourth moment of quadratic Dirichlet $L$-functions over the rational function field.  

Given two Hecke eigenforms $f$ and $g$ (with weights that are congruent modulo $4$) it is a very interesting problem to understand averages of 
$ L(\frac 12, f\otimes \chi_{d}) L(\tfrac 12, g\otimes \chi_{d})$.  An asymptotic formula or lower bound for this quantity could be used to show that 
there are quadratic twists for which $L(\tfrac 12, f\otimes \chi_d)$ and $L(\tfrac 12, g\otimes \chi_d)$ 
are both non-zero; a result that is as yet unknown.  Unfortunately the methods of this 
paper do not shed any light on this problem.




\section{Basic tools}

\noindent  In this section we gather some of the standard formulas and estimates we shall need.

\subsection{The approximate functional equation}

\noindent Let $d$ be a fundamental discriminant and let
$s$ be a complex number in the critical strip.  We define (for any positive $c$) 
\begin{equation*}
W_{w}(x) := \frac{1}{2\pi i} \int_{(c)} \frac{\Gamma(w+\tfrac{\kappa - 1}{2} + s)}{\Gamma(w+ \tfrac{\kappa - 1}{2} )} 
(2\pi x)^{-s} 
\frac{ds}{s}.
\end{equation*}
A particular case is when $w=1/2$ where we 
have 
\begin{equation*}
 W_{1/2}(x) =  \frac{1}{2\pi i} \int_{(c)}  \frac{g(s)}{s} x^{-s} ds, \quad \text{where} \quad g(s) = (2 \pi )^{-s} \frac{\Gamma(\frac{\kappa}{2} + s)}{\Gamma(\frac{\kappa}{2})}.
\end{equation*}
We also set 
\begin{equation*}
{\mathcal{A}}(s,d):= \sum_{n=1}^{\infty} 
\frac{\lam_f(n) \chi_d(n)}{n^s} W_s\left(\frac{ n}{|d|}\right). 
\end{equation*}
The function $W_{w}(x)$ decays rapidly as $x\to \infty$; this may be checked 
by taking $c$ suitably large in the definition of $W_w(x)$, and using Stirling's 
formula.  

\begin{mylemma}  
\label{lem:AFE}
With notations as above we have
that 
\begin{equation*}
L(s,f\otimes\chi_d) 
= {\mathcal{A}}(s,d) + i^{\kappa} \epsilon(d) 
\left(\frac{|d|}{2\pi}\right)^{1-2s} 
\frac{\Gamma(1-s)}{\Gamma(s)} {\mathcal{A}}(1-s,d).
\end{equation*}
\end{mylemma}

Lemma \ref{lem:AFE} is a standard ``approximate functional equation,"  see for example Theorem 5.3 of 
\cite{IK}.  Note that if $s=1/2$ then $L(\half, f \otimes \chi_d) = (1 + i^{\kappa} \epsilon(d)) \mathcal{A}(\half, d)$.

\subsection{Poisson summation}
We now quote Lemma 2.6 of \cite{Sound}.
\begin{mylemma}
\label{lemma:Poisson}
 Let $F$ be a smooth function with compact support on the positive real numbers, and suppose that $n$ is an odd integer.  Then
\begin{equation*}
 \sum_{(d,2) = 1} \leg{d}{n} F\left(\frac{d}{Z}\right) = \frac{Z}{2n} \leg{2}{n} \sum_{k \in \mz} (-1)^k G_k(n) \widehat{F}\left(\frac{kZ}{2n}\right),
\end{equation*}
where
\begin{equation*}
 G_k(n) = \left(\frac{1-i}{2} + \leg{-1}{n} \frac{1+i}{2}\right) \sum_{a \shortmod{n}} \leg{a}{n} \e{ak}{n},
\end{equation*}
and
\begin{equation*}
 \widehat{F}(y) = \intR (\cos(2 \pi x y) + \sin(2 \pi x y)) F(x) dx
\end{equation*}
is a Fourier-type transform of $F$.
\end{mylemma}

The Gauss-type sum $G_k(n)$ has been calculated explicitly in Lemma 2.3 of \cite{Sound}
which we quote below.

\begin{mylemma}
\label{lemma:Gk}
If $m$ and $n$ are relatively prime odd integers, then $G_k(mn) = G_k(m) G_k(n)$, and if $p^{\alpha}$ is the largest power of $p$ dividing $k$ (setting $\alpha=\infty$ if $k=0$), then
\begin{equation*}
\label{eq:Gauss}
 G_k(p^{\beta}) = 
\begin{cases}
 0, \qquad & \text{if $\beta \leq \alpha$ is odd}, \\
 \phi(p^{\beta}), \qquad & \text{if $\beta \leq \alpha$ is even}, \\
 -p^{\alpha}, \qquad & \text{if $\beta = \alpha + 1$ is even}, \\
 \leg{kp^{-\alpha}}{p} p^{\alpha} \sqrt{p}, \qquad & \text{if $\beta = \alpha +1$ is odd}, \\
 0, \qquad & \text{if $\beta \geq \alpha +2$.}
\end{cases}
\end{equation*}
\end{mylemma}

\subsection{The large sieve for quadratic characters}
Heath-Brown \cite{H-B} proved the following large-sieve type inequality for quadratic characters.
\begin{mytheo}
\label{thm:H-B}
 For any $M, N \geq 1$ and any sequence of complex numbers $a_n$, we have
\begin{equation*}
 \sumstar_{\substack{m \leq M\\  m\ \text{\rm odd}}} \Big| \sumstar_{n \leq N} a_n \leg{n}{m} \Big|^2 \ll (MN)^{\varepsilon} (M + N) \sum_{n \leq N} |a_n|^2.
\end{equation*}
\end{mytheo}

Using the approximate functional equation Lemma \ref{lem:AFE}, and Heath-Brown's result 
we may easily deduce the following estimate (a simple modification of Theorem 2 of \cite{H-B}).

\begin{mycoro}
\label{coro:HB}
 For $\sigma \ge \frac 12$, and any $\varepsilon>0$, we have 
\begin{equation*}
\label{eq:H-B}
 \sumflat_{|d|\le X} |L(\sigma + it, f \otimes \chi_d)|^2 \ll_{\varepsilon} (X(1+|t|))^{1 + \varepsilon}.
\end{equation*}
\end{mycoro}

\section{The Main Proposition} 
\label{section:mainprop} 

In this section we describe the main calculation that leads 
to the proof of Theorem \ref{thm:lowerbound}.  The pattern of this proof is also 
followed, with a few modifications, in obtaining a stronger result 
leading to Theorem \ref{thm:GRHasymp}; we shall describe this in Section \ref{section:asymponGRH}.  

Our aim here is to establish an asymptotic formula 
for 
\[ 
S(h):= \sumstar_{(d,2)= 1} \sum_{n_1} \sum_{n_2} \frac{\lam_f(n_1) \lam_{f}(n_2)}{\sqrt{n_1n_2}} \chi_{8d}(n_1n_2) h(d,n_1,n_2),
\] 
where $h$ is a smooth function on ${\mr}_+^3$.

\begin{myprop}
\label{mainprop} 
Let $X$, $U_1$ and $U_2$ be large, and suppose that $U_1 U_2 \le X^2$.  
Let $h(x,y,z)$ be a smooth function 
on $\mr_+^3$ which is compactly supported in the $x$-variable, having all partial 
derivatives extending continuously to the boundary, and satisfying the partial derivative 
bounds  
$$ 
x^i y^j z^k h^{(i,j,k)}(x,y,z) \ll_{i,j,k} \left( 1+\frac xX\right)^{-100} \left(1+ \frac{y}{U_1}\right)^{-100} 
\left(1+ \frac{z}{U_2}\right)^{-100}. 
$$ 
Then, setting $h_1(y,z) = \int_0^{\infty} h(xX,y,z)dx$, 
$$ 
S(h) = \frac{4X}{\pi^2} 
\sum_{\substack{(n_1 n_2,2)=1 \\ n_1 n_2 = \square}} \frac{\lambda_f(n_1) \lambda_f(n_2)}{\sqrt{n_1 n_2}} \prod_{p|n_1n_2} \left( \frac{p}{p+1}\right) 
h_1 \left(n_1, n_2\right) +O((U_1 U_2)^{\frac 14} X^{\frac 12+\varepsilon}) .
$$
\end{myprop}

We begin the proof of Proposition \ref{mainprop} by using M\"{o}bius inversion to remove the squarefree condition on $d$.  Thus we write, for an appropriate parameter $Y$ to be chosen later, 
\begin{eqnarray*}
 S(h) &=& \Big(\sum_{\substack{a \leq Y \\ (a,2)=1}} + \sum_{\substack{a > Y \\ (a,2)=1}} \Big) \mu(a)  \sum_{(d,2)=1} \sum_{(n_1,a)=1} \sum_{(n_2,a)=1} \frac{\lambda_f(n_1) \lambda_f(n_2)}{\sqrt{n_1 n_2}} \chi_{8d}(n_1n_2)h(da^2, n_1, n_2)\\
 &=&S_1(h)+ S_2(h).
\end{eqnarray*}


\subsection{Estimating $S_2(h)$}  
\label{section:S2}
We first estimate the easier term $S_2(h)$. 

\begin{mylemma} 
\label{lem:3.1}
We have
$ S_2(h) \ll X^{1 + \varepsilon} Y^{-1}$.
\end{mylemma}
\begin{proof}  We write $d=b^2 \ell$ where $\ell$ is square-free, and 
group terms according to $c=ab$.  Thus 
\begin{equation}
\label{eq:S21}
 S_2(h) = \sum_{(c,2)=1} \sum_{\substack{a > Y \\ a|c}} \mu(a)  
 \sumstar_{(\ell,2)=1} \sum_{(n_1,c)=1} \sum_{(n_2,c)=1} 
 \frac{\lambda_f(n_1) \lambda_f(n_2) }{\sqrt{n_1 n_2}}
 \chi_{8\ell}(n_1 n_2) h(c^2 \ell, n_1, n_2).
\end{equation}
Consider the sum over $\ell$, $n_1$, and $n_2$ in \eqref{eq:S21}.  
Using Mellin transforms in the variables $n_1$ and $n_2$ we 
see that this sum is 
\begin{equation} 
\label{eq:S22}
\frac{1}{(2\pi i)^2} \int_{(\frac 12+\varepsilon)} \int_{(\frac 12+\varepsilon)} 
\sumstar_{(\ell, 2) =1} {\check h}(c^2 \ell; u, v) \sum_{\substack{n_1, n_2 \\ (n_1n_2, c)=1}} 
\frac{\lam_f(n_1)\lam_f(n_2)}{n_1^{\frac 12+u}n_2^{\frac 12+v} } \chi_{8\ell}(n_1) \chi_{8\ell}(n_2) du \, dv,
\end{equation}
where 
\begin{equation*}
{\check h}(x;u,v) = \int_0^{\infty} \int_0^{\infty} h(x,y,z) y^u z^v \frac{dy}{y} \frac{dz}{z}.
\end{equation*}
Integrating by parts several times, we see that for Re$(u)$, Re$(v) >0$, 
\begin{equation}
\label{eq:3.12}
{\check h}(x;u,v) \ll \left( 1 + \frac{x}{X} \right)^{-100} \frac{U_1^{\text{Re}(u)} U_2^{\text{Re}(v)}}{|uv| (1 + |u|)^{10} (1+ |v|)^{10}}.
\end{equation}
The sum over $n_1$ and $n_2$ in \eqref{eq:S22} equals $L_c(\frac 12+u,f\otimes \chi_{8\ell}) L_c(\frac 12+v,f \otimes \chi_{8\ell})$ where $L_c$ is given by the Euler product defining $L(s,f)$ 
but omitting those primes dividing $c$.   Thus moving the lines of integration in \eqref{eq:S22}
to $\text{Re}(u)=\text{Re}(v)=1/\log X$, and using \eqref{eq:3.12} together 
with 
$$
|L_c(\tfrac 12+u,f\otimes \chi_{8\ell})L_c(\tfrac 12+v,f\otimes \chi_{8\ell})| 
\le d(c)^2 ( |L(\tfrac 12+u,f\otimes \chi_{8\ell})|^2 + |L(\tfrac 12+v,f\otimes\chi_{8\ell})|^2),
$$
we conclude that \eqref{eq:S22} is bounded by 
\begin{equation}
\label{eq:S23}
d(c)^2 (\log X)^2 \int_{-\infty}^{\infty} (1+|t|)^{-10} \sumstar_{(\ell,2)=1} \left(1 + \frac{\ell c^2}{X}\right)^{-100} 
|L(\tfrac 12+\tfrac 1{\log X} +it, f\otimes \chi_{8\ell})|^2  \ dt.
\end{equation}
Now using Corollary \ref{eq:H-B} we conclude that the quantity in \eqref{eq:S22} is 
$\ll d(c)^2 X^{1+\varepsilon}/c^2$, and using this estimate in \eqref{eq:S21} we obtain 
the Lemma.   
\end{proof}

Now let us consider the harder problem of evaluating $S_1(h)$.  
We begin by applying Poisson summation, Lemma \ref{lemma:Poisson}.
Letting $C = \cos$ and $S = \sin$, we get
\begin{multline}
\label{eq:S1}
 S_1(h)= \frac{X}{2} \sum_{\substack{a \leq Y \\ (a,2)=1}} \frac{\mu(a)}{a^2} \sum_{k \in \mz}  
\sum_{(n_1,2a)=1} \sum_{(n_2,2a)=1} \frac{\lambda_f(n_1) \lambda_f(n_2)}{\sqrt{n_1 n_2}} \frac{G_k(n_1 n_2)}{n_1 n_2} 
\\
\times \int_0^{\infty} h(xX, n_1, n_2) (C + S)\leg{2\pi k xX}{2n_1 n_2 a^2} dx.
\end{multline}

\subsection{The main term}
The main contribution to $S_1(h)$ comes from the $k=0$ term in \eqref{eq:S1}, which 
we call $S_{10}(h)$.
Note $G_0(m) = \phi(m)$ if $m = \square$, and is zero otherwise.  Further 
$$
\sum_{\substack{a \leq Y \\ (a,2n_1n_2)=1}} \frac{\mu(a)}{a^2} = \frac{1}{\zeta(2)} 
\prod_{p|2n_1n_2} \left( 1-\frac{1}{p^2}\right)^{-1} +O(Y^{-1}) = \frac{8}{\pi^2}\prod_{p|n_1n_2} \left(1-\frac{1}{p^2}\right)^{-1}+ O(Y^{-1}) .
$$
Hence, setting $h_1(y,z) = \int_0^{\infty} h(xX,y,z) dx$ we obtain that 
\begin{multline*}
 S_{10}(h) =  \frac{4X}{\pi^2} 
\sum_{\substack{(n_1 n_2,2)=1 \\ n_1 n_2 = \square}} \frac{\lambda_f(n_1) \lambda_f(n_2)}{\sqrt{n_1 n_2}} \prod_{p|n_1n_2} \left( \frac{p}{p+1}\right) 
h_1\left( n_1, n_2\right)
\\
 + O \Big( \frac XY \sum_{\substack{(n_1 n_2,2)=1 \\ n_1 n_2 = \square}} 
\frac{d(n_1) d(n_2)}{\sqrt{n_1 n_2}} |h_1(n_1,n_2)|\Big).
\end{multline*}
Using the bounds for $h$ assumed in Proposition \ref{mainprop} (which basically restrict $n_1$ 
to be of size $U_1$ and $n_2$ to be of size $U_2$) we find that the error term above 
is $\ll X (\log X)^{11}/Y$ so that 
\begin{equation}
\label{eq:mainterm} 
S_{10}(h) =  \frac{4X}{\pi^2} 
\sum_{\substack{(n_1 n_2,2)=1 \\ n_1 n_2 = \square}} \frac{\lambda_f(n_1) \lambda_f(n_2)}{\sqrt{n_1 n_2}} \prod_{p|n_1n_2} \left( \frac{p}{p+1}\right) 
h_1\left( n_1, n_2\right) + O\left( \frac XY (\log X)^{11} \right). 
\end{equation}

\subsection{The $k \neq 0$ terms}
\label{section:3.3}
We now estimate the contribution to $S_1(h)$ from the terms  $k \neq 0$ in \eqref{eq:S1}; call 
this contribution $S_3(h)$. 

We first  express the weight function appearing in \eqref{eq:S1} in a form 
more suitable for Mellin transforms.  
Suppose $f$ is a smooth function on $\mr_+$ with 
rapid decay at infinity, and such that $f$ and all its derivatives have a finite limit as $x\to 0^+$.  
Consider the Fourier-like transform 
\begin{equation*}
 \widehat{f}_{CS}(y) := \int_0^{\infty} f(x) CS(2\pi xy) dx,
\end{equation*}
where $CS$ stands for either $\cos$ or $\sin$.  By Mellin inversion, we get 
\begin{eqnarray*}
 \widehat{f}_{CS}(y)  &=& \int_0^{\infty}CS(2\pi xy) \frac{1}{2\pi i} \int_{(-\half)} \widetilde{f}(1+s) x^{-s} ds \frac{dx}{x}\\
&=& \int_0^{\infty}CS(\text{sgn}(y) x) \frac{1}{2\pi i} \int_{(\half)} \widetilde{f}(1-s) (x/2\pi |y|)^{s} ds \frac{dx}{x}.
\end{eqnarray*}
Reversing the orders of integration and using (17.43.3, 17.43.4) of \cite{GR}, the above simplifies to
\begin{equation*}
 \widehat{f}_{CS}(y)  = \frac{1}{2\pi i} \int_{(\half)} \widetilde{f}(1-s) \Gamma(s) CS\left(\frac{\text{sgn}(y) \pi s}{2}\right) (2\pi |y|)^{-s} ds.
\end{equation*}
One can rigorously justify the interchange of integrations by splitting the $x$-integral into $x \leq Z$ (with $Z$ large) and $x > Z$; one treats the integral with $x > Z$ by a contour shift, while the $x \leq Z$ integral gets interchanged with the $s$-integral, and then extended to all $x \geq 0$ by integration by parts.

Applying this formula, we have 
\begin{multline}
\label{eq:3.30}
 \int_0^{\infty} h \left(Xx, n_1, n_2 \right) (C + S)\leg{2\pi k xX}{2n_1 n_2 a^2} dx 
\\
=  \frac{X^{-1}}{2\pi i}  \int_{(\varepsilon)} \check{h}\left(1-s; n_1, n_2 \right) \leg{n_1 n_2 a^2}{\pi |k| }^{s} \Gamma(s) (C + \text{sgn}(k)S)\left(\frac{\pi s}{2} \right) ds, 
\end{multline}
where
\begin{equation*}
\label{eq:3.31}
 \check{h}(s;y,z) = \int_0^{\infty} h(x,y,z) x^s \frac{dx}{x}. 
\end{equation*}
Taking the Mellin transforms in the other variables on the second line of \eqref{eq:3.30}, we get
\begin{equation*}
\frac 1X \leg{1}{2\pi i}^3  \int_{(\varepsilon)} \int_{(\varepsilon)} \int_{(\varepsilon)} \widetilde{h}\left(1-s,u,v \right) \frac{1}{n_1^{u} n_2^{v}} \leg{n_1 n_2 a^2}{\pi |k| }^{s} \Gamma(s) (C + \text{sgn}(k)S)\left(\frac{\pi s}{2} \right) ds\, du \, dv,
\end{equation*}
where 
\begin{equation*}
\widetilde{h}(s,u,v) = \int_{\mr_{+}^{3}} h(x,y,z) x^{s} y^{u} z^{v} \frac{dx}{x} \frac{dy}{y} \frac{ dz}{ z}.
\end{equation*}
Integrating by parts several times we find that for Re$(u)$, Re$(v) >0$ we have 
\begin{equation} 
\label{eq:h}
|\widetilde {h}(s,u,v)| \ll \frac{X^{\text{\rm Re}(s)} U_1^{\text{\rm Re}(u)} U_2^{\text{\rm Re}(v)}}{|uv| (1+|s|)^{98} (1+|u|)^{98} (1 + |v|)^{98}}.
\end{equation}

Using this expression in \eqref{eq:S1}, and since $G_k(m) =G_{4k}(m)$ for odd $m$, 
we find that
\begin{multline}
\label{eq:S31}
 S_3(h) = \frac{1}{2} \sum_{\substack{a \leq Y \\ (a,2)=1}} \frac{\mu(a)}{a^2} \sum_{k \neq 0}  
\sum_{(n_1,2a)=1} \sum_{(n_2,2a)=1} \frac{\lambda_f(n_1) \lambda_f(n_2)}{\sqrt{n_1 n_2}} 
\frac{G_{4k}(n_1 n_2)}{n_1 n_2} 
\\
\leg{1}{2\pi i}^3  \int_{(\varepsilon)} \int_{(\varepsilon)} \int_{(\varepsilon)} \widetilde{h}\left(1-s,u,v \right) \frac{1}{n_1^{u} n_2^{v}} \leg{n_1 n_2 a^2}{\pi |k|  }^{s} \Gamma(s) (C + \text{sgn}(k)S)\left(\frac{\pi s}{2} \right) ds \, du \, dv.
\end{multline}
We write $4k=k_1k_2^2$ where $k_1$ is a fundamental discriminant, and $k_2$ is 
positive, so that the sum over $k$ above is a sum over fundamental discriminants $k_1$ and 
positive integers $k_2$.    We consider the sum over $k_2$, $n_1$ and $n_2$ in \eqref{eq:S31} 
above.  Note that the integrals in \eqref{eq:S31} may be taken over any vertical 
lines with real part between $0$ and $1$.   
Consider 
\begin{equation}
\label{eq:Z}
 Z(\alpha,\beta,\gamma;q,k_1) = \sum_{k_2=1}^{\infty} \sum_{(n_1,2q)=1} \sum_{(n_2,2q)=1} \frac{\lambda_f(n_1) \lambda_f(n_2)}{n_1^{\alpha} n_2^{\beta} |k_2|^{2\gamma}} \frac{G_{k_1 k_2^2}(n_1 n_2)}{n_1 n_2}, 
\end{equation}
which converges absolutely if Re$(\alpha)$, Re$(\beta)$, and Re$(\gamma)$ are 
all $> \frac 12$.  Therefore taking the integrals in \eqref{eq:S31} to be on the lines 
Re$(s)= \frac 12+ \varepsilon$ and ${\text {Re}}(u)={\text {Re}}(v) = \frac 12+\varepsilon$ 
we find that 
\begin{multline*}
S_3(h)  = \frac 12 \sum_{\substack{a \leq Y \\ (a,2)=1}} \frac{\mu(a)}{a^2} 
 \sumflat_{k_1} \left(\frac{1}{2\pi i}\right)^3  \int_{(\frac 12+2\varepsilon)} \int_{(\frac 12+\varepsilon)} 
\int_{(\frac 12+\varepsilon)} {\widetilde h}(1-s,u,v)  \Gamma(s)
\\
\hskip 1 in \times   (C+ {\text {sgn}}(k_1) S) \left( \frac{\pi s}{2}\right)
 \left(\frac{a^2}{\pi |k_1|}\right)^s 
Z(\tfrac 12+u-s, \tfrac 12+v-s, s;a,k_1) ds \, du \, dv.
\end{multline*}
Changing variables we conclude that 
  \begin{multline}
  \label{eq:S32}
S_3(h)  =\frac 12 \sum_{\substack{a \leq Y \\ (a,2)=1}} \frac{\mu(a)}{a^2} 
 \sumflat_{k_1} \left(\frac{1}{2\pi i}\right)^3  \int_{(\varepsilon)} \int_{(\varepsilon)} 
\int_{(\frac 12+\varepsilon)} {\widetilde h}(1-s,u+s,v+s)  \Gamma(s)
\\
\hskip 1in \times   (C+ {\text {sgn}}(k_1) S) \left( \frac{\pi s}{2}\right)
 \left(\frac{a^2}{\pi |k_1|}\right)^s 
Z(\tfrac 12+u, \tfrac 12+v, s;a,k_1) ds \, du \, dv.
\end{multline}
To proceed further we require an analysis of the function $Z$.   

\begin{mylemma}
\label{lemma:Z}
 The function $Z(\alpha,\beta,\gamma;q,k_1)$ defined above may be written 
 as 
 \begin{multline*}
 \frac{L_q(\frac 12+\alpha, f\otimes \chi_{k_1})L(\frac 12+\beta,f\otimes \chi_{k_1})}{\zeta_q(1+\alpha+\beta) L_q(1+2\alpha, \text{\rm sym}^2 f)L_q(1+\alpha+\beta, \text{\rm sym}^2 f) L_q(1+2\beta, \text{\rm sym}^2 f)} Z_2(\alpha,\beta,\gamma;q, k_1),
\end{multline*}
where $Z_2(\alpha,\beta,\gamma;q,k_1)$ is a function uniformly bounded in the region
$\text{\rm Re}(\gamma) \ge \frac 12+\varepsilon$, and $\text{\rm Re}(\alpha), \text{\rm Re}(\beta) \geq 0$.
\end{mylemma}
\begin{proof} 
Inspecting Lemma \ref{lemma:Gk}, we see that the summand of \eqref{eq:Z} is jointly multiplicative in terms of $n_1, n_2$, and $k_2$, so that $Z(\alpha,\beta,\gamma;q, k_1)$ may be 
expressed as a product over all primes $p$.  We must compute the contribution of such an 
Euler factor at $p$.  

Consider first the generic case when 
$p \nmid 2q k_1$.  The contribution of such an Euler factor is 
\begin{equation*}
 \sum_{k_2, n_1, n_2} \frac{\lambda_f(p^{n_1})\lambda_f(p^{n_2})}{p^{n_1 \alpha + n_2 \beta +2k_2 \gamma}}\frac{ G_{k_1 p^{2k_2}}(p^{n_1 + n_2})}{p^{n_1+n_2}  }.
\end{equation*}
In the region Re$(\gamma) \ge \frac 12+\varepsilon$, Re$(\alpha)$, Re$(\beta)\ge 0$  
we check using Lemma \ref{lemma:Gk} that the terms $k_2\ge 1$ contribute terms of size 
$\ll 1/p^{1+2\varepsilon}$.  This leaves the contribution of the 
term $k_2=0$ which is $1+ \chi_{k_1}(p) \lam_f(p) (p^{-\frac 12-\alpha}+p^{-\frac 12-\beta})$.   From 
this calculation we see that this Euler factor for $Z$ matches the corresponding 
Euler factor in the alternative expression given in our Lemma.

Next consider the case $p | k_1$ but  $p \nmid 2q$.  Using Lemma \ref{lemma:Gk} we find that in the region 
Re$(\gamma)\ge \frac 12+\varepsilon$ and Re$(\alpha)$, Re$(\beta)\ge 0$ we have that 
this Euler factor equals 
$$ 
1- \lam_f(p^2) \left( \frac{1}{p^{1+2\alpha}} +\frac{1}{p^{1+\alpha+\beta}} + \frac{1}{p^{1+2\beta}} \right) 
- \frac{1}{p^{1+\alpha+\beta}} + O\left( \frac{1}{p^{1+\varepsilon}}\right). 
$$
Again this matches the corresponding Euler factor prescribed in our Lemma. 


Finally, if $p|2q$ the corresponding Euler factor is $(1-p^{-2\gamma})^{-1}= (1+O(1/p^{1+2\varepsilon})$.   
With these computations we have verified the Lemma.  
\end{proof}

With this information about $Z$ at hand, we return to \eqref{eq:S32}.  We split that sum into 
two terms based on whether $|k_1| \le U_1 U_2 Y^2/X$, or not.   
For the first category of terms we move the lines of integration to Re$(s)= \frac 34$, 
Re$(u)=\text{Re}(v)=-\frac 12+\frac{1}{\log X}$, and for the second category we move 
the lines of integration to Re$(s)=\frac 54$, Re$(u)=\text{Re}(v)=-\frac 12+\frac{1}{\log X}$.  
In either case we find by Lemma \ref{lemma:Z} that 
$$
Z(\tfrac 12+u,\tfrac 12+v,s;a,k_1) 
\ll |L_a(1+u,f\otimes \chi_{k_1}) L_a(1+v,f\otimes \chi_{k_1})| (\log X)^4 
$$
which is 
\begin{equation}
\label{eq:Zbound}
\ll (\log X)^4 \prod_{p|a} \left( 1+ \frac{10}{\sqrt{p}}\right) \left(|L(1+u,f\otimes \chi_{k_1})|^2 + 
|L(1+v,f\otimes \chi_{k_1})|^2\right).
\end{equation}

Using \eqref{eq:h}, \eqref{eq:Zbound}, that $|\Gamma(s) (C+\text{sgn}(k_1)S)(\pi s/2)| \ll |s|^{\text{Re}(s)-\frac 12}$, and the symmetry in $u$ and $v$
we find that our first category of terms 
contributes 
\begin{multline}
\label{eq:firstbd}
\ll (XU_1 U_2)^{\frac 14} (\log X)^4 \sum_{a\le Y} \frac{1}{\sqrt{a}} \prod_{p|a} \left(1+\frac{10}{\sqrt{p}}\right)  \int_{(\frac 34)} \int_{(-\frac 12+\frac{1}{\log X})}\int_{(-\frac 12+\frac{1}{\log X})} \\ 
\times \sumflat_{|k_1|\le U_1 U_2 Y^2/X}\frac{1}{|k_1|^{\frac 34}} |L(1+u,f\otimes \chi_{k_1})|^2 \frac{ du \, dv \, ds}{(1+|s|)^{50}(1+|u+s|)^{50}(1+|v+s|)^{50}}.
\end{multline}
Using Corollary \ref{eq:H-B} we conclude that the above is $\ll (U_1 U_2)^{\frac 12}Y X^{\varepsilon}$.

Similarly the contribution of the second category of terms is 
\begin{multline} 
\label{eq:secbd}
\ll (U_1 U_2)^{\frac 34} X^{-\frac 14} (\log X)^4 \sum_{a\le Y} \sqrt{a} \prod_{p|a} \left(1+\frac{10}{\sqrt{p}}\right)  \int_{(\frac 54)} \int_{(-\frac 12+\frac{1}{\log X})}\int_{(-\frac 12+\frac{1}{\log X})} \\ 
\times \sumflat_{|k_1|> U_1 U_2 Y^2/X}\frac{1}{|k_1|^{\frac 54}} |L(1+u,f\otimes \chi_{k_1})|^2 \frac{ du \, dv \, ds}{(1+|s|)^{50}(1+|u+s|)^{50}(1+|v+s|)^{50}}.
\end{multline}
Using 
Corollary \ref{eq:H-B} again we see that the above is once again 
$\ll (U_1 U_2)^{\frac 12} Y X^{\varepsilon}$.  

The proof of Proposition \ref{mainprop} follows upon 
combining the work of the previous sections, choosing $Y= X^{\frac 12} (U_1 U_2)^{-\frac 14}$.  


\section{ The lower bound: Proof of Theorem \ref{thm:lowerbound}}
\label{section:lowerbound}

 
\noindent  Let $F$ be a smooth, nonnegative, compactly supported function on $\mr_{+}$.  
 Define
\begin{equation*}
 \mathcal{A}_U(\tfrac12;8d) = 2 \sum_{n=1}^{\infty} \frac{\lambda_f(n) \chi_{8d}(n)}{\sqrt{n}} W\leg{n}{U},
\end{equation*}
where $W(x) = W_{1/2}(x)$, and $U\le X$ is a parameter that we shall choose shortly.
Since $F$ is nonnegative, by Cauchy's inequality we have that 
\begin{equation}
\label{eq:lower}
\sumstar_{(d,2)=1} L(\tfrac12, f \otimes \chi_{8d})^2 F \left( \frac {8d}{X} \right)\geq  \frac{\left(\sumstar_{(d,2)=1} L(\tfrac12, f \otimes \chi_{8d}) \mathcal{A}_U(\tfrac12;8d) F\leg{8d}{X} \right)^2}{\sumstar_{(d,2)=1} \mathcal{A}_U(\tfrac12;8d)^2 F\leg{8d}{X}}.
\end{equation}
Write the right hand side above as $(4A)^2/4B$, say.  Using Proposition \ref{mainprop} 
we shall be able to evaluate $A$ and $B$ asymptotically in the range $U\le X^{1-\varepsilon}$.  
If we choose $U = X^{1-\varepsilon}$ then both $4A$ and $4B$ will be close 
to the expected asymptotic for the second moment, giving the lower bound of Theorem \ref{thm:lowerbound}.  
A similar truncation argument appeared in \cite{Sfourth}.


Both $A$ and $B$ may be written in  a form suitable for applying Proposition \ref{mainprop}. 
For example, we have 
\begin{equation}
\label{eq:4.2}
A= \sumstar_{(d,2)=1} \sum_{n_1} \sum_{n_2} \frac{\lambda_f(n_1) \lambda_f(n_2) \chi_{8d}(n_1 n_2)}{\sqrt{n_1 n_2}} h(d, n_1, n_2),
\end{equation}
where 
$h(x,y,z) = F(8x/X) W(y/U) W(z/8x)$, which satisfies the hypothesis in Proposition 
3.1 with $X=X$, $U_1=U$ and $U_2=X$.   A similar expression holds 
for $B$ with $h(x,y,z) = F(8x/X) W(y/U) W(z/U)$ meeting the condition in Proposition \ref{mainprop} 
with $X=X$, $U_1=U_2=U$.   

If $U\le X^{1-\varepsilon}$ then applying Proposition \ref{mainprop} in \eqref{eq:4.2} 
we find that 
\begin{equation}
\label{eqn:4.3}
A = \frac{4X}{\pi^2} \sum_{\substack {(n_1n_2,2)=1 \\ n_1 n_2= \square}} \frac{\lam_f(n_1) \lam_f(n_2)}{\sqrt{n_1n_2}} \prod_{p|n_1 n_2} \left(\frac{p}{p+1} \right) \int_0^{\infty} h(xX,n_1,n_2) dx + o(X). 
\end{equation} 
Using the definition of $W(x)$ we obtain that 
$$ 
\int_0^\infty h(xX,n_1,n_2) dx = 
\frac{1}{(2\pi i)^2} \int_{(1)} \int_{(1)} \frac{g(u)g(v)}{uv} \frac{U^u X^v}{n_1^u n_2^v} \frac {{\tilde F}(1+v)}{8} du \, dv, 
$$ 
where ${\tilde F}(1+v) = \int_0^{\infty} x^{v} F(x) dx$.  Using this in \eqref{eqn:4.3}  
and setting 
$$ 
Z(u,v)  = \sum_{\substack{(n_1 n_2,2)=1 \\ n_1 n_2 = \square}} \frac{\lambda_f(n_1) \lambda_f(n_2)}{n_1^{\frac12 + u} n_2^{\frac12 + v}}\prod_{p|n_1 n_2} \left( \frac{p}{p+1}\right),
$$ 
we conclude that 
\begin{equation}
\label{eqn:4.4}  
A = \frac{X}{2\pi^2} \frac{1}{(2\pi i)^2} \int_{(1)} \int_{(1)} \frac{g(u)g(v)}{uv} U^u X^v {\tilde F}(1+v) Z(u,v) 
dv \ du + o(X).
\end{equation}

 A simple calculation shows that $Z(u,v)$ equals 
 \begin{multline}
 \label{eq:Z(u,v)}
  \prod_{p >2} \left( 1+ \frac{p}{p+1} \left(\frac{1}{2} \left(1-\frac{\lam_f(p)}{p^{\frac 12+u}}+ \frac{1}{p^{1+2u}}\right)^{-1} \left( 1-\frac{\lam_f(p)}{p^{\frac 12+v}} +\frac{1}{p^{1+2v}}\right)^{-1} \right. \right. \\ 
  \left. \left. + 
 \frac12 \left(1+\frac{\lam_f(p)}{p^{\frac 12+u}}+\frac{1}{p^{1+2u}} \right)^{-1} 
 \left( 1+ \frac{\lam_f(p)}{p^{\frac 12+v}} +\frac{1}{p^{1+2v}} \right)^{-1} -1 \right) \right).
 \end{multline} 
The Euler product above converges absolutely when Re$(u)$ and Re$(v)$ are positive.  
We write 
\begin{equation}
\label{eq:Z(u,v)2} 
\zeta(1+u+v) L(1+2u, \text{sym}^2(f)) L(1+2v, \text{sym}^2(f)) L(1+u+v, \text{sym}^2(f))  Z_2(u,v),
\end{equation}
where $Z_2(u,v)$ converges absolutely in the region Re$(u)$ and Re$(v)$ larger than $-\frac 14+\varepsilon$, and is uniformly bounded there.
 
 We now use these observations to evaluate the double integral in \eqref{eqn:4.4}.  First 
 we move the integrals there to Re$(u)= \text{Re}(v) =\frac{1}{10}$; no poles are 
 encountered in this shift.  Then we move the line of integration in $v$ to Re$(v) = -\frac 15$.   
 In doing so we encounter simple poles at $v=0$ and $v=-u$ whose residues we next calculate;  
 the integrals on  Re$(u)=\frac{1}{10}$, Re$(v)=-\frac 15$ are easily seen to be $O(X^{-\frac 1{10}+\varepsilon})$.  The contribution from the residue at $v=-u$ is 
 $$
 \frac{1}{2\pi i} \int_{(\frac 1{10})} \frac{g(u)g(-u)}{-u^2} U^u X^{-u} L(1+2u, \text{sym}^2 f)
 L(1-2u, \text{sym}^2 f) L(1, \text{sym}^2 f) Z_2(u,-u) {\tilde F}(1-u) du
 $$ 
which is $O(1)$ since $U \le X$.  Finally consider the contribution of the 
residue at $v=0$, namely
$$ 
\frac{1}{2\pi i}\int_{(\frac 1{10})} \frac{g(u)}{u}  U^u \zeta(1+u)L(1+2u,\text{sym}^2 f) L(1+u,\text{sym}^2 f) 
L(1,\text{sym}^2 f) Z_2(u,0) {\tilde F}(1)du.
$$
 We now move the line of integration in $u$ to Re$(u)=-\frac 15$, encountering 
a double pole at $u=0$, and  
the integral on the $-\frac 15$ line contributes $\ll U^{-\frac 15}X^{\varepsilon}$.  
The residue of the double pole at $u=0$ is easily seen to be 
$$ 
{\tilde F}(1) L(1,\text{sym}^2 f)^3 Z_2(0,0) \log U + O(1). 
$$  

Using these observations in \eqref{eqn:4.4} 
we conclude that 
\begin{equation}
\label{Aequals}
A = \frac{X}{2\pi^2} \left( {\tilde F}(1) L(1,\text{sym}^2 f)^3 Z_2(0,0) \log U + O(1)\right). 
\end{equation} 
A similar argument shows that the same asymptotic holds for $B$.  
Choosing $U=X^{1-\varepsilon}$, and letting $F$ approximate the indicator 
function of $[0,1]$, the inequality \eqref{eq:lower} gives Theorem \ref{thm:lowerbound}.

\section{The asymptotic on GRH: Proof of Theorem \ref{thm:GRHasymp}}
\label{section:asymponGRH}

\noindent In this section we prove Theorem \ref{thm:GRHasymp}.   The key ingredient in the 
proof is the following  upper bound on shifted moments whose proof we postpone to the 
next section.  

\begin{mycoro}
\label{shiftbounds}  Assume  GRH for the family of quadratic twists of $f$, and 
the for the Riemann zeta-function, and for the symmetric square $L$-function 
$L(s,\text{\rm sym}^2 f)$ .   Let $t_1$ and $t_2$ be real numbers 
with $|t_1|, |t_2|\le X$ and let $\tfrac 12 \le \sigma \le \frac 12+ \frac{1}{\log X}$.  
Then 
$$
\sumflat_{|d|\le x} |L(\sigma+it_1,f\otimes \chi_d) L(\sigma+it_2,f\otimes \chi_d)| 
\ll X(\log X)^{\frac 12+\varepsilon} \left(1+   \min( (\log X)^{\frac 12}, |t_1-t_2|^{-\frac 12})\right).
$$
\end{mycoro}

Let $F$ be a smooth, nonnegative, compactly supported function $\mr_+$.  
Recall from the previous section the definition of ${\mathcal A}_U(\frac 12;8d)$, and 
write $L(\frac 12,f\otimes \chi_{8d}) = {\mathcal A}_U(\frac 12;8d)+ {\mathcal B}_U(\frac 12;8d)$.  
We shall prove that, on GRH, for $U\le X/(\log X)^{100}$ we have 
\begin{equation} 
\label{Aest}
\sumstar_{(d,2)=1} |{\mathcal A}_U(\tfrac 12;{8d})|^2 F\left(\frac{8d}{X}\right)= 
\frac{2X}{\pi^2} \left( {\tilde F}(1) L(1,\text{sym}^2 f)^3 Z_2(0,0) \log U + O(1) \right), 
\end{equation} 
and 
\begin{equation}
\label{Best} 
\sumstar_{(d,2) =1} |{\mathcal B}_U(\tfrac 12; 8d)|^2 F\left(\frac{8d}{X}\right) = O\left( X (\log X)^{\frac 12+\varepsilon} (\log X/U)^2 \right). 
\end{equation}
Once these two estimates are established, Theorem \ref{thm:GRHasymp} will follow upon 
choosing $U=X/(\log X)^{100}$ since 
\begin{multline*} 
\sumstar_{(d,2)=1} |L(\tfrac 12,f\otimes \chi_{8d})|^2 F\left(\frac{8d}{X}\right)=
\sumstar_{(d,2)=1} |{\mathcal A}_U(\tfrac 12;8d)|^2 F\left(\frac{8d}{X}\right) 
\\
+ O\left( \sumstar_{(d,2)=1} \left(|{\mathcal A}_U(\tfrac 12;8d){\mathcal B}_U(\tfrac 12;8d)| +|{\mathcal B}_U(\tfrac 12; 8d)|^2 \right)F\left(\frac{8d}{X}\right)\right),
\end{multline*}
and using \eqref{Aest} and \eqref{Best} together with Cauchy-Schwarz we see that 
the remainder term above is $O(X (\log X)^{\frac 34+\varepsilon})$.

It remains now to prove \eqref{Aest} and \eqref{Best}.  We start with the latter.   From the definition 
of ${\mathcal B}_U$ we find that 
\begin{equation*}
{\mathcal B}_U(\tfrac 12;8d) = \frac{1}{\pi i} \int_{(c)} g(s) L(\tfrac 12+s,f \otimes \chi_{8d}) 
\left( \frac{(8d)^s - U^s}{s} \right) ds. 
\end{equation*} 
Since $((8d)^s-U^s)/s$ is analytic for all $s$ we may move the line of integration 
above to the line Re$(s)=0$, and since $|(8d)^s-U^s|/|s| \ll |\log (8d/U)|$ we find that 
$$ 
|{\mathcal B}_U(\tfrac 12;8d)| \ll |\log (8d/U)| \int_{-\infty}^{\infty} |g(it)| |L(\tfrac 12+it,f\otimes \chi_{8d})| dt. 
$$ 
Therefore we find that the RHS in \eqref{Best} is
$$ 
\ll (\log X/U)^2 \int_{-\infty}^{\infty} \int_{-\infty}^{\infty} |g(it_1)g(it_2)| \sumstar_{(d,2)=1} 
|L(\tfrac 12+it_1,f\otimes \chi_{8d})L(\tfrac 12+it_2,f\otimes \chi_{8d})| dt_1 \, dt_2. 
$$
If $|t_1|$ and $|t_2|$ are both below $X/2$ then we use Corollary \ref{shiftbounds} to bound the 
inner sum over $d$ above.  In the remaining case we use Cauchy's 
inequality and the bound of Corollary \ref{coro:HB}.   Since $|g(t)|$ decreases exponentially 
in $|t|$, an easy calculation then gives \eqref{Best}.  

Now we turn to \eqref{Aest}.   The argument follows the pattern laid out in Sections 
\ref{section:mainprop} and \ref{section:lowerbound}; the results there may be improved by appealing to the estimates of 
Corollary \ref{shiftbounds} in place of the weaker estimate in Corollary \ref{coro:HB}.   Consider 
Proposition \ref{mainprop}; on GRH we claim that the remainder term there 
can be replaced with $O((U_1 U_2)^{\frac 14} X^{\frac 12} (\log X)^{20})$.   
We give only the changes that need to be made to the argument there.   In Section \ref{section:S2}
we use Corollary \ref{shiftbounds} in place of Corollary \ref{coro:HB} (in the range $|t|\le X$)  to estimate the quantity in \eqref{eq:S23}.  
This shows that the bound in Lemma \ref{lem:3.1} may be replaced with $X(\log X)^{10}Y^{-1}$.  
Similarly in Section \ref{section:3.3}, in estimating \eqref{eq:firstbd} and \eqref{eq:secbd} we again invoke Corollary \ref{shiftbounds} 
(in the range $|u| \le X$) to obtain there the improved bound of $(U_1U_2)^{\frac 12} Y (\log X)^5$.  
Choosing $Y=X^{\frac 12}(U_1 U_2)^{-\frac 14}$ as before, we obtain the stated 
bound for the remainder term in Proposition \ref{mainprop}.    The 
argument of \S \ref{section:lowerbound} now goes through verbatim establishing \eqref{Aest}.


\section{Upper bounds for shifted moments assuming GRH}
\label{section:shifted}

\noindent Given a real number $x\ge 10$ and a complex number $z$, set
 \begin{equation*}
\label{eq:L}
 \mathcal{L}(z,x) 
=
	\begin{cases}
	 \log \log{x}, \qquad & |z| \le (\log{x})^{-1}, \\
	 -\log{|z|}, \qquad &(\log{x})^{-1} \le |z| \le 1, \\
	0, \qquad & 1 \le |z|. \\
	\end{cases}
\end{equation*}
If $z_1$ and $z_2$ are complex numbers we define 
\begin{equation*} 
{\mathcal M}(z_1, z_2, x) = -\tfrac 12 \left( {\mathcal L}(z_1,x) +{\mathcal L}(z_2,x)\right), 
\end{equation*}
and 
\begin{multline*}
{\mathcal V}(z_1,z_2,x) = \tfrac 12 \left( {\mathcal L}(2z_1,x) + {\mathcal L}(2z_2,x) 
+ {\mathcal L}(2\text{Re}z_1,x) + {\mathcal L}(2\text{Re}z_2,x) \right. \\
 \left. + 2{\mathcal L}(z_1+z_2,x) + 2 {\mathcal L}(z_1 +\overline{z_2},x)\right).
\end{multline*} 
This section is devoted to establishing, on GRH, the following estimates 
for shifted moments of $L$-functions.  An immediate consequence of 
this theorem is Corollary \ref{shiftbounds} which we used above to establish Theorem \ref{thm:GRHasymp}. 

\begin{mytheo}
\label{thm:GRHupperbound} 
Let $X$ be large, and let $z_1$ and $z_2$ be two complex numbers with 
$0\le \text{\rm Re}(z_1), \text{\rm Re}(z_2) \le 1/\log X$ and with $|z_1|, |z_2| \le X$.  Assume GRH 
for the family of quadratic twists of $f$, and for the Riemann zeta-function, 
and for the symmetric square $L$-function $L(s,\text{\rm sym}^2 f)$.  Then 
for any positive real number $k$ and any $\varepsilon >0$ we have
$$ 
\sumflat_{|d|\le X} |L(\tfrac 12 +z_1, f\otimes \chi_d)L(\tfrac 12+z_2, f\otimes \chi_d)|^k 
\ll_{k,\varepsilon} X (\log X)^{\varepsilon} \exp\big( k{\mathcal M}(z_1,z_2,X) +\tfrac{k^2}{2} {\mathcal V}(z_1,z_2,X)\big). 
$$
\end{mytheo}

As remarked earlier this result follows upon modifying the method of \cite{SoundMoment}, and 
similar results for the Riemann zeta-function were obtained by V. Chandee \cite{Chandee}.  
If we set $z_1=z_2 = it$ for a real number $t$, then it is expected that when $t$ is close 
to zero the moments correspond to a family with ``orthogonal" symmetry, while 
for larger $t$ (for example $t=1$) the expected symmetry type is ``unitary."  We note 
that our Theorem above expresses in a uniform way the transition between 
these symmetry types.  

To prove Theorem \ref{thm:GRHupperbound}, we shall establish an estimate on the frequency with which 
large values of $|L(\tfrac 12+z_1,f\otimes \chi_d)L(\tfrac 12+z_2,f\otimes \chi_d)|$ 
are attained.   As $d$ varies over the discriminants of 
size below $X$ we expect that   $|L(\tfrac 12+z_1,f\otimes \chi_d)L(\tfrac 12+z_2,f\otimes \chi_d)|$ is 
distributed normally with mean ${\mathcal M}(z_1,z_2,X)$ and variance ${\mathcal V}(z_1,z_2,X)$. 
The next Proposition establishes (in a range sufficient to prove Theorem \ref{thm:GRHupperbound}) 
an upper bound for the frequency of large values that conforms to 
the above prediction.  In what follows it may be helpful to keep in mind that 
for $z_1$ and $z_2$ as in Theorem \ref{thm:GRHupperbound} the quantity ${\mathcal V}(z_1,z_2,X)$ 
lies between $\log \log X+ O(1)$ and $4\log \log X+O(1)$.

\begin{myprop}
\label{prop:measure}  
With assumptions as in Theorem \ref{thm:GRHupperbound}, let
${\mathcal N}(V; z_1,z_2, X)$ denote the number of fundamental disciminants $|d| \leq X$ 
such that $\log{|L(\tfrac12 + z_1, f \otimes \chi_d)L(\tfrac 12+z_2, f\otimes \chi_d)|} 
\geq V + {\mathcal M}(z_1,z_2,X)$.  
In the range $10 \sqrt{\log \log{X}} \leq V \leq {\mathcal V}(z_1,z_2,X)$ we have
\begin{equation*}
\label{eq:measure1}
{\mathcal N}(V;z_1,z_2,X) \ll X \exp\left(-\frac{V^2}{2{\mathcal V}(z_1,z_2,X)} \left(1 - \frac{25}{\log \log \log {X}}\right)\right);
\end{equation*}
for ${\mathcal V}(z_1,z_2,X) < V \leq \frac{1}{16}{\mathcal V}(z_1,z_2,X) \log \log \log  X$ we have
\begin{equation*}
 {\mathcal N}(V;z_1,z_2,X) \ll X \exp\left(-\frac{V^2}{2{\mathcal V}(z_1, z_2,X)} \left(1 - \frac{15V}{ \mathcal{V}(z_1,z_2,X) \log \log \log {X}}\right)^2\right);
\end{equation*}
finally, for $\frac{1}{16}\mathcal{V}(z_1,z_2,X) \log \log \log X < V$ we have
\begin{equation*}
 {\mathcal N}(V;z_1,z_2,X) \ll X \exp\left(-\frac{1}{1025} V \log{V}\right).
\end{equation*}
\end{myprop}

We now show how Theorem \ref{thm:GRHupperbound} may be deduced from Proposition \ref{prop:measure}.  
\begin{proof}[Proof of Theorem \ref{thm:GRHupperbound}]  We have 
\begin{multline*}
\sumflat_{|d|\le X} |L(\tfrac 12+z_1,f\otimes \chi_d)L(\tfrac 12+z_2,f\otimes \chi_d)|^k 
= -\int_{-\infty}^{\infty} \exp({kV}+k{\mathcal M}(z_1,z_2,X)) d{\mathcal N}(V;z_1,z_2,X) 
\\
= k\int_{-\infty}^{\infty} \exp({kV}+k{\mathcal M}(z_1,z_2,X)) {\mathcal N}(V;z_1,z_2,X) dV.
\end{multline*}
Inserting here the bounds for ${\mathcal N}(V;z_1,z_2,X)$ furnished by 
Proposition \ref{prop:measure}  we obtain Theorem \ref{thm:GRHupperbound} with a little calculation.  This calculation 
may be facilitated by using Proposition \ref{prop:measure} in the crude form ${\mathcal N}(V;z_1,z_2, X) 
\ll X (\log X)^{o(1)} \exp(-V^2/(2{\mathcal V}(z_1,z_2,X)))$ for $3\le V\le 4k{\mathcal  V}(z_1,z_2,X)$ 
and that ${\mathcal N}(V,z_1,z_2,X) \ll X (\log X)^{o(1)} \exp(-4k V)$ for larger $V$.  
\end{proof}

It remains now to prove Proposition \ref{prop:measure}.  We first obtain an auxiliary result analogous to the Proposition of \cite{SoundMoment}, namely \eqref{eq:6.4} below.  Write $\lam_f(p) = \alpha_p + \beta_p$ where $\alpha_p \beta_p =1$ and, 
by Deligne's theorem, $|\alpha_p|=|\beta_p|=1$.  
By modifying slightly the proof of the Proposition in \cite{SoundMoment} we obtain that 
for  $z_1$, $z_2$, $d$ as in Theorem \ref{thm:GRHupperbound}, and for any $2\le x\le X$, 
\begin{multline}
\label{eq:6'}
 \log |L(\tfrac12 + z_1, f \otimes \chi_d)L(\tfrac 12+z_2,f\otimes \chi_d)| \leq \text{Re}\Big(\sum_{
 \substack{p^l \leq x \\  l \geq 1}}  \frac{\chi_d(p^l) (\alpha_p^l + \beta_p^l)}{l p^{l(\half + \frac{\lambda_0}{\log{x}})} }(p^{-l z_1}+p^{-l z_2}) \frac{\log(x/p^l)}{\log{x}}\Big) \\
+ 2(1+\lam_0) \frac{\log X}{\log x} + O\left(\frac{1}{\log{x}}\right),
\end{multline}
where $\lam_0 =0.4912 \ldots$ is the unique real number satisfying $e^{-\lam_0} = \lam_0 +\lam_0^2/2$.  As in \cite{SoundMoment}, the terms with $l \geq 3$ give $O(1)$.  
Using that $\alpha_p^2 +\beta_p^2 = \lam_f(p^2) -1$, and that $\sum_{p|d} 1/p \ll \log \log \log X$ the terms with $l=2$ give 
\begin{equation} 
\label{eq:6.3}
 \text{Re} \sum_{p\le \sqrt{x}} \frac{\lam_f(p^2)-1}{2p^{1+2\frac{\lam_0}{\log x} }} 
(p^{-2z_1} + p^{-2z_2}) \frac{\log (x/p^2)}{\log x} + O(\log \log \log X).
\end{equation}
Using the GRH for $L(s,\text{sym}^2 f)$ we may see that 
\begin{equation}
\label{eqn:GRH1}
\sum_{p \le y} (p^{-2z_1} + p^{-2z_2}) \lam_f(p^2) \log p \ll \sqrt{y} (\log Xy)^2,
\end{equation}
and also the sum is trivially $\ll y$.  From these bounds and partial summation we obtain 
that 
$$
\sum_{p\le \sqrt{x}}   \frac{\lam_f(p^2)}{p^{1+2\frac{\lam_0}{\log x} }} 
(p^{-2z_1} + p^{-2z_2}) \frac{\log (x/p^2)}{\log x} = O(\log \log \log X).
$$ 
Similarly RH gives that 
\begin{equation}
\label{eqn:GRH2} 
\sum_{p\le y} (p^{-2z_1}+ p^{-2z_2}) \log p = \frac{y^{1-2z_1}}{1-2z_1} + \frac{y^{1-2z_2}}{1-2z_2} 
+ O\left( \sqrt{y} (\log Xy)^2\right), 
\end{equation}
and once again the sum is also trivially $\ll y$.   Partial summation now shows that 
$$
-\frac 12 \sum_{p\le \sqrt{x}}   \frac{1}{p^{1+2\frac{\lam_0}{\log x} }} 
(p^{-2z_1} + p^{-2z_2}) \frac{\log (x/p^2)}{\log x} = {\mathcal M}(z_1,z_2,x)+ O(\log \log \log X). 
$$ 
Inserting the above estimates into \eqref{eq:6.3} and \eqref{eq:6'}, 
and since ${\mathcal M}(z_1,z_2,x) 
\le {\mathcal M}(z_1,z_2,X) + \log X/\log x$, we conclude that 
\begin{multline} 
\label{eq:6.4}
\log |L(\tfrac 12+z_1,f\otimes \chi_d) L(\tfrac 12+z_2, f\otimes \chi_d)| 
\le \text{Re} \sum_{2<p\le x} \frac{\lam_f(p)\chi_d(p)}{p^{\frac 12 +\frac {\lam_0}{\log x}}} 
(p^{-z_1} + p^{-z_2} ) \frac{\log x/p}{\log x} \\
+ {\mathcal M}(z_1,z_2,X) + 4 \frac{\log X}{\log x} + O( \log \log \log X). 
\end{multline} 




\begin{mylemma}
\label{lemma:quadDirichletpol} Let $X$ and $y$ be real numbers and $k$ a natural 
number with $y^k \leq X^{1/2}/\log{X}$.   For any complex numbers $a(p)$ we have
\begin{equation*}
 \sumflat_{|d| \leq X} \Big| \sum_{2<p \leq y} \frac{a(p) \chi_d(p)}{p^{\frac12}} \Big|^{2k} \ll 
X \frac{(2k)!}{k! 2^k} \left(\sum_{p \leq y} \frac{|a(p)|^2}{p} \right)^k,
\end{equation*}
where the implied constant is absolute.
\end{mylemma}
\begin{proof} Expanding out and using the P{\' o}lya-Vinogradov 
inequality, we have 
\begin{multline*}
 \sumflat_{|d| \leq X} \Big| \sum_{2<p \leq y} \frac{a(p) \chi_d(p)}{p^{\frac12}} \Big|^{2k}
\leq
\sum_{|d| \leq X} \Big| \sum_{2<p_1, \dots, p_k \leq y} \frac{a(p_1) \dots a(p_k) \ }{\sqrt{p_1 \dots p_k}} 
\leg{d}{p_1\cdots p_k} \Big|^2
\\
\ll
X \sum_{\substack{p_1, \dots, p_{2k} \leq y \\ p_1 \dots p_{2k} = \square}} \frac{|a(p_1) \dots a(p_{2k}) |}{\sqrt{p_1 \dots p_{2k}}} + O\left(\sum_{p_1, \dots, p_{2k}} |a(p_1) \dots a(p_{2k}) \log({y^{2k}})| \right).
\end{multline*}
Since $y^{2k} \log(y^{2k}) \ll X$ we see, using Cauchy-Schwarz, that the second 
term above is 
\begin{equation*}
 \ll \log(y^{2k}) \left(\sum_{p \leq y} |a(p)|\right)^{2k}  \leq \log(y^{2k})  \left(\sum_{p \leq y} \frac{|a(p)|^2}{p} \right)^k \left( \sum_{p \leq y} p \right)^k \leq X \left(\sum_{p \leq y} \frac{|a(p)|^2}{p} \right)^k.
\end{equation*}
To estimate the first term, note that $p_1 \dots p_{2k} = \square$ precisely 
when there is  a way to pair up the indices so that the corresponding primes are equal.   
There are $\frac{(2k)!}{k!2^k}$ ways in which 
the $2k$ indices may be paired up.  Hence
\begin{equation*}
 \sum_{\substack{p_1, \dots, p_{2k} \leq y \\ p_1 \dots p_{2k} = \square}} \frac{|a(p_1) \dots a(p_{2k}) |}{\sqrt{p_1 \dots p_{2k}}} \leq \frac{(2k)!}{k! 2^k} \left(\sum_{p} \frac{|a(p)|^2}{p} \right)^k. \qedhere
\end{equation*}
\end{proof}

\begin{proof} [Proof of Proposition \ref{prop:measure}]
For brevity put ${\mathcal V}={\mathcal V}(z_1,z_2,X)$, and set
\begin{equation*}
 A =
	\begin{cases}
	  \half \log \log \log X, \qquad &V \leq  {\mathcal V}, \\
	  \frac{{\mathcal V}}{2V} \log\log \log X, \qquad &  {\mathcal V}< V \leq \frac{1}{16} {\mathcal V} \log\log \log X, \\
	  8, \qquad &V > \frac{1}{16} {\mathcal V} \log \log \log X.
	\end{cases}
\end{equation*}
Define further $x = X^{A/V}$, and $z = x^{1/\log \log{X}}$.  

By taking $x =\log X$ in 
\eqref{eq:6.4} and bounding the sum over $p$ trivially, we may assume $V \leq \frac{5\log{X}}{\log \log{X }}$.  Then by \eqref{eq:6.4} we have
\begin{equation*}
\log |L(\tfrac12 + z_1, f \otimes \chi_d)L(\tfrac 12+z_2,f\otimes \chi_d)| \leq S_1 + S_2 +\mathcal{M}(z_1,z_2,X) + 5\frac{V}{A},
\end{equation*}
where $S_1$ is the sum there truncated to $p \leq z$, and $S_2$ is the sum over $z < p \leq x$.  If $d$ is such that $\log |L(\tfrac12 + z_1, f \otimes \chi_d)L(\tfrac 12+z_2,f\otimes \chi_d)| \geq V +  \mathcal{M}(z_1,z_2,X)$, then either
\begin{equation*}
 S_2 \geq \frac{V}{A}, \qquad \text{or} \qquad S_1 \geq V\left(1- \frac{6}{A}\right) =: V_1.
\end{equation*}

By Lemma \ref{lemma:quadDirichletpol} we see that for any $k \leq \frac{V}{2A} -1$ we have
\begin{equation*}
 \sumflat_{|d| \leq X} |S_2|^{2k} \ll X \frac{(2k)!}{k!2^k} \left(\sum_{z < p \leq x} \frac{4}{p}\right)^k  \ll 
 X(3k\log \log \log X)^k.
\end{equation*}
Hence, choosing $k=\lfloor V/(2A) \rfloor -1$ and with 
a little calculation, the number of discriminants $|d|\le X$ with $S_2 \geq V/A$ is
\begin{equation*}
 \ll X \exp\left(-\frac{V}{4A} \log{V}\right).
\end{equation*}

We now seek a bound for the number of discriminants with $S_1$ large.  By Lemma \ref{lemma:quadDirichletpol}, 
we find that for  any $k \leq \log(X^{1/2}/\log{X})/\log{z}$,
\begin{equation*}
 \sumflat_{|d| \leq X} |S_1|^{2k} \ll X \frac{(2k)!}{k! 2^k} \left( \sum_{p \leq z} \frac{a(p)^2}{p}\right)^k,
\end{equation*}
where $a(p) =  \lam_f(p) p^{-\lam_0/\log x} \frac{\log x/p}{\log x} 
\text{Re} (p^{-z_1}+p^{-z_2})$.  
Note that
\begin{equation*}
 \sum_{p \leq z} \frac{a(p)^2}{p} \le 
\frac 14 \sum_{p \leq X} \frac{\lambda_f(p)^2}{p} (p^{-z_1} +p^{-\overline{z_1}} +p^{-z_2} +
p^{-\overline{z_2}} )^2 = {\mathcal V}(z_1,z_2,X) +O(\log \log \log X), 
\end{equation*}
upon using \eqref{eqn:GRH1} and \eqref{eqn:GRH2} and partial summation.  
Thus we get that the number of $|d| \leq X$ such that $S_1 \geq V_1$ is
\begin{eqnarray*}
 &\ll&X V_1^{-2k} \frac{(2k)!}{2^k k!} ({\mathcal V}(z_1,z_2,X)+ O(\log \log \log X))^k 
 \\
 &\ll& X \left(\frac{2k( \mathcal{V}(z_1,z_2,X) + O(\log \log \log X))}{e V_1^2} \right)^k.
\end{eqnarray*}
When $V \leq (\log \log{X})^2$, we take $k$ to be $\lfloor V_1^2/(2{\mathcal V}(z_1,z_2,X)) \rfloor$, 
and for $V > (\log \log{X})^2$  we take $k$ to be $\lfloor 10V \rfloor$.   Then 
the above estimates give that the number of discriminants $|d|\le X$ 
with $S_1 \ge V_1$ is 
\begin{equation*}
 \ll X \exp\left(-\frac{V_1^2}{2{\mathcal V}(z_1,z_2,X)} \left(1 + O\left(\frac{\log \log \log X}{\log \log{X}}\right)\right)\right) + X \exp(-V \log{V}).
\end{equation*}
Combining this estimate with our estimate for the frequency with 
which $S_2$ can be large, we obtain the Proposition. 
\end{proof}

\end{document}